\documentclass[12pt]{amsart}

\title{Graded nilpotent Lie algebras of infinite type}
\author{Boris Doubrov, Olga Radko}

\newcommand{\g}{\mathfrak g}
\newcommand{\gb}{\bar\g}
\newcommand{\h}{\mathfrak h}
\newcommand{\m}{\mathfrak m}
\newcommand{\n}{\mathfrak n}
\newcommand{\gl}{\mathfrak{gl}}
\newcommand{\End}{\operatorname{End}}

\newcommand{\Der}{\operatorname{Der}}
\newcommand{\Aut}{\operatorname{Aut}}
\newcommand{\rank}{\operatorname{rank}}
\newcommand{\im}{\operatorname{Im}}
\newcommand{\ad}{\operatorname{ad}}
\newcommand{\GL}{\mathrm{GL}}
\newcommand{\al}{\alpha}
\newcommand{\om}{\omega}
\newcommand{\C}{\mathbb{C}}
\newcommand{\R}{\mathbb{R}}
\newcommand{\cM}{\mathcal{M}}
\newcommand{\cE}{\mathcal{E}}
\newcommand{\cF}{\mathcal{F}}
\newcommand{\sll}{\mathfrak{sll}}
\newcommand{\spp}{\mathfrak{sp}}
\newcommand{\sym}{\operatorname{sym}}
\newcommand{\cchar}{\operatorname{char}}

\newtheorem*{tancr}{Tanaka criterium}
\newtheorem*{tanthm}{Tanaka theorem}
\newtheorem*{spencr}{Spencer criterium}

\newtheorem{thm}{Theorem}

\newtheorem{cor}{Corollary}

\theoremstyle{remark}
\newtheorem{rem}{Remark}



\address{Belarussian State University, Nezavisimosti av.~4, 220030, Minsk, Belarus}
\email{doubrov@islc.org}

\begin{document}

\begin{abstract}
The paper gives the complete characterization of all graded nilpotent Lie
algebras with infinite-dimensional Tanaka prolongation as extensions of graded
nilpotent Lie algebras of lower dimension by means of a commutative ideal. We
introduce a notion of weak characteristics of a vector distribution and prove
that if a bracket-generating distribution of constant type does not have
non-zero complex weak characteristics, then its symmetry algebra is necessarily
finite-dimensional. The paper also contains a number of illustrative algebraic
and geometric examples including the proof that any metabelian Lie algebra with
a 2-dimensional center always has an infinite-dimensional Tanaka prolongation.
\end{abstract}

\maketitle

\section{Introduction}
This paper is devoted to the study of non-holonomic vector distributions with
infinite-dimensional symmetry algebras. The simplest examples of such
distributions are contact distributions on odd-dimensional manifolds and, more
generally, contact systems on jet spaces $J^k(\R^n,\R^m)$. The symmetry
algebras of such distributions are given by Lie--Backlund theorem and are
isomorphic to either the Lie algebra of all vector fields on
$J^0(\R^n,\R^m)=\R^{n+m}$ for $m\ge 2$ or to the Lie algebra of all contact
vector fields on $J^1(\R^n,\R)$ for $m=1$. In both cases the symmetry algebras
are infinite-dimensional.

We are interested in only so-called \emph{bracket-generating distributions},
i.e., we always assume that repetitive brackets of vector fields lying in $D$
generate the whole tangent bundle $TM$. If this is not the case, then $D$ lies
in a certain proper completely integrable vector distribution $D'$, and the
geometry of $D$ can be essentially reduced to the restrictions of $D$ to the
fibers of $D'$.

We shall also say that the distribution $D$ is \emph{degenerate}, if it
possesses non-zero \emph{Cauchy characterisitics}, i.e., vector fields $X\in D$
such that $[X,D]\subset D$. It is also clear that any degenerate distribution,
even if it is bracket-generating, has an infinite-dimensional symmetry algebra.
For the proof and other basic properties of vector distributions we refer
to~\cite[Chapter 2]{bcggg}.

While it seems to be very difficult to provide the complete local description
of all vector distributions with infinite-dimensional symmetry algebra, it
appears that it is possible to give the complete description of the their
symbol algebras.

Namely, let $D$ be a bracket-generating vector distribution on a smooth
manifold~$M$. Taking repetitive brackets of vector fields lying in $D$, we can
define \emph{a weak derived flag} of $D$:
\begin{gather*}
0 \subset D \subset D^2 \subset \dots \subset D^{\mu} = TM,\\
D^0 = 0,\quad D^{1} = D,\quad D^{i+1} = [D,D^i],\ i\ge 2.
\end{gather*}

At each point $p\in M$ we can defined the associated graded vector space
\begin{equation}\label{symbol}
\m(p) = \sum_{i<0} \m_{-i}(p),\quad \m_{-i}(p) = D_p^{i}/D_p^{i-1}.
\end{equation}
Since $[D^i,D^j]\subset D^{i+j}$ for all $i,j\ge 0$, $\m(p)$ is naturally
equipped with a structure of a graded Lie algebra. Namely, if $x\in \m_{-i}(p)$
and $y\in\m_{-j}(p)$ are two homogeneous elements in $\m(p)$, and $X\in D^i$,
$Y\in D^j$ are two vector fields such that $X_p+D^{i-1}_p=x$ and
$Y_p+D^{j-1}_p=y$, then the value of $[X,Y]+D^{i+j-1}_p$ depends only on $x$
and $y$, and, thus, defines a graded Lie algebra structure on $\m(p)$. It is
also clear from the definition, that $\m(p)$ is a nilpotent Lie algebra
generated by $\m_{-1}(p)$.

This Lie algebra is called a symbol of the distribution $D$ at a point $p\in
M$, and it plays essential role in study of $D$ and any geometric structures
subordinate to $D$. For example, if $D$ is a contact structure on a smooth
manifold of dimension $2n+1$, i.e., a non-degenerate codimension 1 distribution
on $M$, then its symbol is isomorphic to the $(2n+1)$-dimensional Heisenberg
Lie algebra at any point $p\in M$.

The family of graded Lie algebras $\m(p)$ is a basic invariant of any
bracket-generating distribution, which includes not only the dimensions of the
weak derived series of $D$, but also a non-trivial algebraic information. In
many cases the structure of these algebras has very important geometric
consequences and allows to associate various geometric structures with~$M$. One
of the most famous examples is E.~Cartan paper~\cite{car10}, where he
associates a $G_2$-geometry with any non-degenerate 2-dimensional vector
distribution on a 5-dimensional manifold.

We say that the distribution $D$ has constant symbol $\m$ or is of type $\m$ if
its symbols $\m(p)$ are isomorphic to $\m$ for all points $p\in M$. For
example, the contact distribution and all contact systems on jet spaces are of
constant type.

We shall use the term \emph{graded nilpotent Lie algebra} or simply GNLA for
any negatively graded Lie algebra $\m=\sum_{i=1}^{\mu}\m_{-i}$ generated by
$\m_{-1}$. We shall call such Lie algebra \emph{non-degenerate}, if $\m_{-1}$
does not include any non-zero central elements. This has a clear geometric
meaning. It is easy to see that the symbol of a bracket-generating distribution
$D$ is non-degenerate if and only if $D$ has no non-zero Cauchy
characteristics. The largest $\mu$ such that $\m_{-\mu}\ne 0$ is called
\emph{the depth} of $\m$.

It appears that we can derive the exact bound for the dimension of the symmetry
algebra of $D$ purely in terms of its symbol $\m$. Namely, Tanaka~\cite{tan1}
has shown in his pioneer works on the geometry of filtered manifolds, that
there is a well-defined graded Lie algebra $\g(\m)$ called Tanaka (or
universal) prolongation of $\m$. It is characterized by the following
conditions:
\begin{enumerate}
\item $\g_{i}(\m) = \m_i$ for all $i<0$;
\item if $[X,\m]=0$ for certain $X\in \g_i(\m)$, then $X=0$;
\item $\g(\m)$ is the largest graded Lie algebra satisfying the above two
conditions.
\end{enumerate}
The Lie algebra $\g(\m)$ has also a clear geometric meaning. Namely, let $M$ be
a connected simply connected Lie group with the Lie algebra $\m$. For example,
we can identify $M$ with $\m$ and define the Lie group multiplication in $M$ by
means of Campbell--Haussdorf series. Define the distribution $D$ on $M$ by
assuming that it is left invariant and that $D_e = \m_{-1}$. Then $\g(\m)$ can
be naturally identified with the graded Lie algebra associated with the
filtered Lie algebra of all germs of infinitesimal symmetries of $D$ at the
identity. In particular, if $\g(\m)$ is finite or infinite-dimensional, so is
the symmetry algebra of the corresponding filtration. Such distributions $D$
are called \emph{standard} distributions of type $\m$.

One of the main result of Tanaka paper~\cite{tan1} can be reformulated as
follows.
\begin{tanthm}[\cite{tan1}]
If $\g(\m)$ is finite-dimensional, then with each distribution $D\subset TM$ of
type $\m$ we can associate a canonical coframe on a certain bundle $P$ over $M$
of dimension $\dim \g(\m)$. In particular, the symmetry algebra of the
distribution $D$ is finite-dimensional, and its dimension is bounded by $\dim
\g(\m)$.
\end{tanthm}
Thus, if $D$ is a holonomic distribution with infinite-dimensional symmetry
algebra, then the Tanaka prolongation of its symbol $\g(\m)$ should also be
infinite-dimensional. The properties of Tanaka prolongation are also studied
in~\cite{yam93,yatsui,kuzmich,war07,ottazzi}.

The main result of this paper is Theorem~\ref{thm:1}, which gives the the
complete characterization of all GNLA with infinite-dimensional Tanaka
prolongation as extensions of graded nilpotent Lie algebras of lower dimension
by means of a commutative ideal. Along with this description we also introduce
a notion of \emph{weak characteristics} of a vector distribution and prove that
if a bracket-generating distribution of constant type does not have non-zero
complex weak characteristics, then its symmetry algebra is necessarily
finite-dimensional.

This article is closely related to a series of papers devoted to the geometry
of $2$ and $3$-dimensional distributions~\cite{ankru, dist2, dist3}. These
papers show that under very mild non-degeneracy conditions the symmetry algebra
of a non-holonomic vector distribution becomes finite-dimensional. As we show
in Theorem~\ref{thm:2}, the set of all GNLA, whose Tanaka prolongation is
infinite-dimensional, forms a closed subvariety in the variety of all GNLA~$\m$
with fixed dimensions of subspaces $\m_{-i}$, $i>0$.

Finally, in the last section of the paper we present a number of illustrative
algebraic and geometric examples including the proof that any metabelian Lie
algebra with a 2-dimensional center always has an infinite-dimensional Tanaka
prolongation.

We would like to express our gratitude to Ian Anderson, Tohru Morimoto, Ben
Warhurst, Igor Zelenko for stimulating discussions on the topic of this paper.

\section{Tanaka and Spencer criteria}
Below we assume that all Lie algebras we consider are defined over an arbitrary
field $k$ of characteristic $0$. We try to be as generic as possible in our
algebraic considerations and we explicitly state, when we require the base
field $k$ to be algebraically closed.

Let $\m$ be an arbitrary GNLA and let $\g(\m)$ be its Tanaka prolongation. We
say that $\m$ is \emph{of finite (infinite) type}, if $\g(\m)$ is
finite-dimensional (resp., infinite-dimensional). Note that this notion is
stable with respect to the base field extensions, since each subspace
$\g_i(\m)$, $i\ge 0$ of the Tanaka prolongation can be computed by solving  a
system of linear equations. Namely, assuming that all subspaces $\g_i(\m)$,
$i<k$ are already computed, the space $\g_k(\m)$ can be defined as follows:
\begin{multline*}
\g_k(\m) = \{ \phi \colon \m\to \sum_{i<k}\g_i(\m) \mid \phi(\m_{-i})\subset
\g_{k-i}(\m),\\
\phi([x,y])=[\phi(x),y]+[x,\phi(y)], \forall x,y\in\m \}.
\end{multline*}

Tanaka criterium reduces the question whether Tanaka prolongation $\g(\m)$ of
$\m$ is finite-dimensional or not to the same question for the standard
prolongation (as described, for example, in~\cite{stern}) of a certain linear
Lie algebra. Namely, let $\Der_0(\m)$ be the Lie algebra of all
degree-preserving derivations of $\m$. Define the subalgebra $\h_0\subset
\Der_0(\m)$ as follows:
\[
\h_0 = \{d \in \Der_0(\m) \mid d(x)=0\text{ for all } x\in \m_{-i}, i\ge 2\}.
\]
We can naturally identify $\h_0$ with a subspace in $\End(\m_{-1})$. Then
Tanaka criteruim can be formulated as follows:
\begin{tancr}[\cite{tan1}] Tanaka prolongation $\g(\m)$ of $\m$ is finite-dimensional if and
only if so is the standard prolongation of $\h_0 \subset \End(\m_{-1})$.
\end{tancr}

In its turn, Spencer criterium provides a computationally efficient method of
detecting whether the standard prolongation of a linear subspace $A\subset
\End(V)$ is finite-dimensional or not.
\begin{spencr}[\cite{spencer69,gqs}]
The standard prolongation of a subspace $A\subset \End(V)$ is
finite-dimensional if and only if $A^{\bar k}\subset \End(V^{\bar k})$ does not
contain endomorphisms of rank~$1$.
\end{spencr}
Here by $\bar k$ we mean the algebraic closure of our base field $k$, and by
$A^{\bar k}$ the subspace in $\End(V^{\bar k})$ obtained from $A$ by field
extension. A detailed and self-contained proof of Spencer criterium can be
found in~\cite{warottazzi}.

We emphasize that this criterium if computationally efficient, as the set of
all rank 1 endomorphisms in $A$ is described by a finite set of quadratic
polynomials, and, for example, Gr\"obner basis technique provides an algorithm
to determine whether this set of polynomials has a non-zero common root.

In the next section we use both these criteria to prove that all fundamental
Lie algebras $\m$ of infinite type over an algebraically closed field possess a
very special algebraic structure.

\section{Graded nilpotent Lie algebras of infinite type}
Let $\g$ be an arbitrary finite-dimensional Lie algebra, and let $V$ be an
arbitrary $\g$-module. We recall that a Lie algebra $\gb$ is called an
extension of $\g$ by means of $V$, if $\gb$ can be included into the following
exact sequence:
\[
0\to V \to \gb \to \g \to 0.
\]
In other words, $V$ is embedded into $\gb$ as a commutative ideal, the quotient
$\gb/V$ is identified with $\g$ and the natural action of $\g=\gb/V$ on $V$
coincides with the predefined $\g$-module structure on $V$. Two extensions
$\gb_1$ and $\gb_2$ are called equivalent, if there exists an isomorphism
$\gb_1\to\gb_2$ identical on $V$ and $\g\equiv\gb_1/V\equiv\gb_2/V$.

It is well-known that equivalence classes of such extensions are described by
the cohomology space $H^2(\g,V)$. Namely, if $[\alpha]$ is an element of
$H^2(\g,V)$, where $\alpha\in Z^2(\g,V)$, then $\gb$ can be identified as a
vector space with $\g\times V$ with the Lie bracket given by:
\begin{equation}\label{ext}
[(x_1,v_1), (x_2,v_2)] = ([x_1,x_2], x_1.v_2-x_2.v_1+\alpha(x_1,x_2)),
\end{equation}
$x_1,x_2\in\g$, $v_1,v_2\in V$.

Assume now that both the Lie algebra $\g$ and the $\g$-module $V$ are graded,
that is $\g=\sum\g_i$, $V=\sum V_j$ and $\g_i.V_j\subset V_{i+j}$. Then the
cohomology space $H^2(\g,V)$ is naturally turned into the graded vector space
as well:
\[
H^2(\g,V) = \sum H^2_i(\g,V).
\]
It is easy to see that the Lie algebra $\gb$ defined by~\eqref{ext} is also
graded if and only if $[\alpha]\in H_0^2(\g,V)$. Thus, we see that the graded
extensions $\gb$ of the graded Lie algebra $\g$ by means of the graded
$\g$-module $V$ are in one to one correspondence with the elements of the
vector space $H^2_0(\g,V)$.

We can introduce an additional group action on $H^2_0(\g,V)$ as follows. Define
$\Aut_0(\g,V)$ as a subgroup in $\Aut_0(\g)\times\GL_0(V)$ of the form:
\begin{multline*}
\Aut_0(\g,V)=\{(f,g)\in\Aut_0(\g)\times\GL_0(V)\mid
\\ g(x.v)=f(x).g(v),\ \forall x\in\g, v\in V\}.
\end{multline*}
This group naturally acts on $H_0^2(\g,V)$:
\[
(f,g).[\al]=[g\circ\al\circ f^{-1}].
\]
It is easy to see that elements of $H_0^2(\g,V)$ lying in the same orbit of
this action correspond to the isomorphic Lie algebras.

In this paper we shall consider \emph{special extensions of GNLA}. Namely, let
$\n$ be an arbitrary (possibly degenerate) GNLA. Fix any subspace
$W\subset\n_{-1}$ and consider $\n_{-1}/W$ as a graded commutative Lie algebra
concentrated in degree $-1$. Take any graded $\n_{-1}/W$-module
$V=\sum_{i=1}^{\mu}V_{-i}$ generated by $V_{-1}$ (as a module). Assuming that
$W.V=0$ and $n_{-i}.V=0$ for all $i\ge 2$ we can consider $V$ as an
$\n$-module. We call any extension of the GNLA $\n$ by means of such
$\n$-module $V$ \emph{a special GNLA} extension. Let $[\al]$ be the
corresponding element of $H_0^2(\n,V)$. It is easy to see that such extension
is non-degenerate if and only if the following two conditions hold:
\begin{enumerate}
\item the condition $\n.v=0$ implies $v=0$ for $v\in V_{-1}$;
\item the cocycle $\al$ is non-degenerate on $Z(\n)\cap \n_{-1}$.
\end{enumerate}
We note that unlike the extensions of GNLA considered in the
papers~\cite{kuzmich,ankru}, the special extensions defined above are not
central extensions of Lie algebras, as the ideal $V$ above does not in general
belong to the center, and the action of~$\n$ on~$V$ is non-trivial.

The main result of the paper is:
\begin{thm}\label{thm:1}
Let $\m$ be a non-degenerate GNLA oven an algebraically closed field. Then the
following conditions are equivalent:
\begin{enumerate}
\item there exists such $d\in \Der_0(\m)$ that $d(\m_{-i})=0$ for $i\ge 2$ and
$\rank d = 1$;
\item there exists such $y\in\m_{-1}$ that $\rank \ad y = 1$; in particular, in
this case $[y,\m_{-i}]=0$ for all $i\ge 2$;
\item $\m$ can be represented as a special extension of a certain (possibly degenerate) GNLA $\n$
with $\dim \m_{-1}=\dim \n_{-1} + 1$ and $\dim V_{-i} = 1$ for all
$i=-1,\dots,-\mu$ and $\mu\ge 2$;
\item $\m$ is of infinite type.
\end{enumerate}
\end{thm}
\begin{proof}
$(1)\Rightarrow(2)$. Let $\h$ be a subalgebra in $\Der_0(\m)$ defined by
condition $d(\m_{-i})=0$ for any $d\in\h$. Let $d\in\h$ and $\rank d = 1$.
Define $W\subset \m$ as $W=\ker d$. Choose non-zero $x,y\in\m_{-1}$ such that
$d(x) = y$. Then we have $\m_{-1}=\langle x \rangle \oplus W$, and $W \supset
[\m,\m]$. Next,
\[
[y, W] = [d(x), W] = - [x,d(W)] = 0.
\]
Since $\m$ is non-degenerate, this implies that $\rank \ad y = 1$ and $[x,y]\ne
0$. We also have
\[
[x, d(y)] = - [d(x),y]=-[y,y]=0.
\]
On the other hand,
\[
[d(y),W] = -[y,d(W)] = 0.
\]
Hence, $[d(y),\m]=0$, and from non-degeneracy of $\m$ it follows that $d(y)=0$.
In particular, $d$ is nilpotent.

$(2)\Rightarrow(1)$. Let $y\in \m_{-1}$ such that $\rank\ad y = 1$. Let $W=\ker
\ad y$. Choose such $x\in\m_{-1}$ that $[x,y]\ne 0$. Such $x$ always exists by
non-degeneracy of $\m$. It is clear that $\m=\langle x\rangle \oplus W$ and
$y\in W$. It is also clear that $W\supset [\m,\m]$. Define $d\in\h$ as $d(x)=y$
and $d(W)=0$. It is easy to see that $d$ is indeed a differentiation of $\m$.

$(2)\Rightarrow(3)$. Again, let $x,y$ be elements in $\m_{-1}$ such that $\rank
\ad y = 1$, $[x,y]\ne 0$, $\m=\langle x \rangle W$, $W\supset [\m,\m]$ and
$[y,W]=0$. Set $y_1=y$ and define $y_{i+1}=[x,y_{i}]$ for all $i\ge 2$. Note
that $y_2\ne 0$. Let $\mu$ be the smallest integer such that $y_{\mu}\ne 0$ and
$y_{\mu+1}=0$. By induction we get:
\[
[y_{i+1},W]=[[x,y_i],W]=[x,[W,y_i]]+[y_i,[x,W]]=0.
\]
Let $V=\langle y_1,\dots y_{\mu}$. It is clear that $V\subset W$, $[V,W]=0]$
and $[x,V]\subset V$. Hence, $V$ is a commutative ideal in $\m$. Put $\n=\m/V$.
It is easy to see that $V$ is generated by $V_{-1}$ as an $\n$-module and the
action of $[\n,\n]$ on $V$ is trivial. Finally, it is evident that $\n$ is also
generated by $\n_{-1}$ as a graded Lie algebra. Thus, $\m$ is represented as a
special extension:
\begin{equation}\label{eq:specext}
0 \to V \to \m \to \n \to 0.
\end{equation}
$(3)\Rightarrow(2)$. Let $\m$ be a special extension~\eqref{eq:specext} of $\n$
with $\dim V_{-i}=1$, $i=1,\dotsm,\mu$. Take $y$ as a non-zero element in
$V_{-1}$. Since $[\n,\n]$ acts trivially on $V$, we see that $[y,\m]\subset
V_{-2}$. Since $V$ is generated by $V_{-1}$ as an $\n$-module, we see that $\ad
y \ne 0$ and $\rank \ad y = 1$.

$(4)\Leftrightarrow (1)$. This is exactly the combined Tanaka and Spencer
criteria.
\end{proof}

\begin{rem}
Only the last implication uses the fact that the base field is algebraically
closed. In case of arbitrary field of characteristic 0 the items (1), (2) and
(3) are still equivalent and imply that $\m$ has infinite type. a simple
counterexample for the implication $(4)\Rightarrow(1)$ over $\R$ is given by
the 3-dimensional complex Heisenberg Lie algebra (with an obvious grading of
depth 2) viewed as a 6-dimensional real Lie algebra.
\end{rem}
\begin{rem}
Condition~(2) of the Theorem appears already in~\cite{ottazzi} as a sufficient
condition for $\m$ to be of infinite type.
\end{rem}

Item~(3) of Theorem~\ref{thm:1} gives an inductive algorithm for constructing
all GNLA of infinite type. Namely, to construct all GNLA $\m$ of infinite type
and $\dim \m_{-1} = n+1$ one needs to:
\begin{itemize}
\item take an arbitrary GNLA $\n$ (of finite or infinite type) such that
$\dim\n_{-1}=n$;
\item take an arbitrary subspace $W\subset \n_{-1}$ of codimension~$1$;
\item fix any integer $s\ge 2$ and construct a graded $\n_{-1}/W$-module
$V=\sum_{i=1}^s V_{-i}$, which is uniquely defined by the conditions $\dim
V_{-i}=1$, $i=1,\dots,s$ and $V_{-i-1} = (\n_{-1}/W).V_{-i}$;
\item compute cohomology space $H_0^2(\n,V)$, where $V$ is treated as an
$\n$-module with the trivial action of $W\oplus [\n,\n]$;
\item take any $\om\in H^2_0(\n,V)$ and define $\m$ as an extension of $\n$ by means of $V$
corresponding to $\om$.
\end{itemize}

In coordinates this procedure can be reformulated as follows. Fix any convector
$\alpha \in \n_{-1}^*$ and an element $X\in \n_{-1}$ such that $\alpha(X)=1$.
Extend $X$ to the basis $\{X,Z_1,\dots,Z_r\}$ of $\n$ consisting of homogeneous
elements. Define $\m$ as a Lie algebra with a basis
$\{X,Y_1,\dots,Y_s,Z_1,\dots,Z_r\}$, $s\ge 2$, where
\begin{align*}
[X,Y_i] &= Y_{i+1}, \quad i =1,\dots,s-1;\\
[Z_j,Y_i] &= 0,\quad i =1,\dots,s,\ j=1,\dots,r;\\
[Y_i,Y_j] &= 0,\quad i,j=1,\dots,s;\\
[X,Z_i]_{\m} &= [X,Z_i]_{\n} + \sum_{j=1}^s a_i^j Y_j;\\
[Z_i,Z_j]_{\m} &= [Z_i,Z_j]_{\n} + \sum_{k=1}^s b_{ij}^k Y_k,
\end{align*}
where the constants $a_i^j$ and $b_{ij}^k$ define a cocycle $\wedge^2 \n\to V$
of degree $0$ (i.e., $\m$ is a graded Lie algebra) and are viewed modulo the
changes of the basis $X\mapsto X + \sum_{i=1}^s c_i Y_s$, $Z_i\mapsto
Z_i+\sum_{j=1}^s d_i^j Y_j$. See the next section for the concrete examples.

Theorem~\ref{thm:1} provides also an easy way to prove that the set of all
graded Lie algebras of infinite type forms forms a closed algebraic subvariety
in the variety of all graded nilpotent Lie algebras. Namely, fix a sequence of
integers $k_1,\dots,k_\mu$ and denote by $M(k_1,\dots,k_\mu)$ a variety of all
graded nilpotent Lie algebras $\m$ such that $\dim \m_{-i} = k_i$ for
$i=1,\dots,\mu$ and $\m_{-i}=0$ for $i>\mu$. It is easy to see that it is a
closed algebraic variety in the vector space of all skew-symmetric graded
algebras (without any conditions on multiplication).

The set of all fundamental graded Lie algebras in $M(k_i)$ is distinguished by
conditions $[\m_{-1},\m_{-i}]=\m_{-i-1}$ for all $i=1,\dots,\mu-1$. It is easy
to see that these conditions define an open subset (in Zarisski topology) in
$M(k_i)$, which we shall denote by $MF(k_i)$.

Define $MF^\infty(k_i)$ as the set of all fundamental Lie algebras of infinite
type.
\begin{thm}\label{thm:2}
$MF^\infty(k_i)$ is a closed subvariety in $MF(k_i)$.
\end{thm}
\begin{proof}
The existence of an element $y\in\m_{-1}$ such that $\rank \ad y = 1$ (and,
hence, $[y,\m_{-i}]=0$ for $i\ge 2$ can be reformulated as an existence of a
non-zero solution of a certain system of quadratic equations (2 by 2 minors of
the matrix of $\ad y$), whose coefficients are structure constants of the Lie
algebra $\m$. It is well-known~\cite[Theorem~3.12]{harris} that that the
conditions when such non-zero solution exists are given in terms of the
algebraic equations on the coefficients of this system.
\end{proof}

Theorem~\ref{thm:1} also has a number of immediate geometric applications.
Namely, let $D$ be a bracket generating distribution with a constat symbol
$\m$. As usual, we assume that $D$ has no Cauchy characteristics. This implies
that $\m$ is non-degenerate, i.e., does not have non-zero central elements
lying in $\m_{-1}$.

We say that a vector field $Y\in D$ is \emph{a weak characteristic of $D$}, if
the following conditions hold:
\begin{enumerate}
\item $Y\in \cchar(D^i)$ for all $i\ge 2$, or, in other words, $[Y, (D^i)]\subset
(D^i)$ for all $i\ge 2$;
\item the subbundle $D'\subset D$ spanned by all $X\in D$ such that
$[X,Y]\subset D$ has codimension $1$ in $D$.
\end{enumerate}
Similarly, we can define \emph{complex weak characteristics} of a vector
distribution $D$ as sections of the comlpexified subbundle $D^{\C}$ of the
comlpexified tangent bundle $T^{\C}M$ satisfying the conditions (1) and (2)
above, where we replace powers $D^i$ by their complexified versions.

\begin{thm}\label{thm:3}
Let $D$ be a bracket generating distribution with a constant symbol $\m$.
Assume that $D$ has no Cauchy characteristics and that $\dim \sym(D)=\infty$.
Then $D$ contains non-trivial complex week characteristics.
\end{thm}
\begin{proof}
This immediately follows from item~(2) of Theorem~\ref{thm:1} and the Tanaka
theorem~\cite{tan1} that bounds the dimension of $\sym (D)$ by the dimension of
the Tanaka prolongation of its symbol $\m$.
\end{proof}
\begin{cor}
If $D \cap (\cap_{i\ge 2} \cchar (D^i)) = 0$, then the symmetry algebra of $D$
is finite-dimensional.
\end{cor}
\begin{proof} Assume that $\sym(D)$ is infinite-dimensional. Then it possesses
complex week characteristic $Y$. Clearly both $\overline Y$ and $c Y$, $c\in
C$, are also complex week characteristics. Therefore, the complex subbundle
spanned by $Y$ and $\overline Y$ is stable with respect to the complex
conjugation. Hence, it also contains a real one-dimensional subbundle spanned
by a certain non-zero vector field $X$. While $X$ itself might no longer be a
week characteristic, it still lies in $D \cap (\cap_{i\ge 2} \cchar (D^i))$.
\end{proof}
\begin{rem}
We formulate Theorem~\ref{thm:2} and its corollary only for the vector
distribution with constant symbol, since the Tanaka theorem on the bound of
$\dim \sym(D)$ has been proved up to now only in this case. The Tanaka proof
does not immediately generalize to the case of non-constant symbol, and this
still remains an open problem according to the authors knowledge.
\end{rem}

\section{Examples}
\subsection{Symbols of 2-dimensional distributions of infinite type}
As a first application of Theorem~\ref{thm:1} we describe the symbols of
2-dimensional bracket generating distributions of infinite type.

Using Theorem~\ref{thm:1}, it is easy to prove that in each dimension $n\ge 2$
there exists a unique (up to isomorphism) GNLA $\m$ with $\dim \m_{-1}=2$. It
turns out to coincide with the symbol of the Goursat distribution, which can
coincides also with the canonical contact system on the jet space
$J^{n-2}(\R,\R)$.

Indeed, by Theorem~\ref{thm:1} any such GNLA must be a special extension of a
certain GNLA $\n$ with $\dim \n_{-1}=1$. But then $\n=\R$ is a one-dimensional
commutative Lie algebra, and all its special extensions are trivial. In each
dimension there is exactly one such extension. It can be described as a Lie
algebra $\m$ with a basis $\langle X, Z_1, \dots, Z_{n-1}\rangle$ and the only
non-zero Lie brackets $[X,Z_i]=Z_{i+1}$, $i=1,\dots,n-2$. Here $\deg X = -1$
and $\deg Z_i=-i$, $i=1,\dots,n-1$. It is easy to see that the contact system
on $J^{n-2}(\R,\R)$ is equivalent to the standard distribution of type $\m$.

\subsection{Symbols of 3-dimensional distributions of infinite type}
In case of $\dim \m_{-1}=3$ we already get a lot of various examples, including
the examples of non-trivial special extension of GNLA. As above, any GNLA $\m$
of infinite type with $\dim \m_{-1}=3$ is a special extension of a certain GNLA
$\n$ with $\dim \n_{-1}=2$. Note that $\n$ itself does not need to be of
infinite type. As the classification of all such GNLA unknown at the moment, it
makes also impossible to classify all infinite type GNLA with $\dim \m_{-1}=3$.

As an example, let us describe all special extensions in case, when $\n$ is a
three-dimensional Heisenberg algebra. Then all trivial special extensions of
$\n$ can be described as:
\[
\m = \langle X, Z_1, Z_2, Y_1, \dots, Y_k \rangle,
\]
where $\deg X = -1$, $\deg Y_i=-i$, $i=1,\dots,k$, $\deg Z_j=-j$, $j=1,2$. All
non-zero Lie brackets are $[X,Z_1]=Z_2$ and $[X,Y_i]=Y_{i+1}$, $i=1,\dots,k-1$.
The Lie algebra $\m$ is a semidirect product of its subalgebra $\n=\langle X,
Z_1, Z_2\rangle$ and the commutative ideal $V=\langle Y_1,\dots,Y_k\rangle$.
Geometrically the standard distribution of type $\m$ can be described as a
canonical contact system on the mixed jet bundle $J^{1,k-1}(\R,\R^2)$.

Any non-trivial special extension of $\n$ can be obtained by adding non-trivial
Lie brackets of the form:
\begin{align*}
[X,Z_1] &= Z_2 + a Y_2;\\
[X,Z_2] &= b Y_3;\\
[Z_1,Z_2] &= c Y_3.
\end{align*}
By changing $Z_1$ to $Z_1-a Y_1$ and $Z_2$ to $Z_2-b Y_2$ we can always obtain
$a=b=0$. If $a\ne 0$, then the Jacobi identity is satisfied only if $k=3$. In
this case we can always scale $a$ to $1$. Thus, we see that up to the
isomorphism there exists exactly one non-trivial special extension of the
three-dimensional Heisenberg Lie algebra. It is given by:
\[
\m = \langle X, Z_1, Z_2, Y_1, Y_2, Y_3 \rangle,
\]
where non-zero Lie brackets are given by:
\begin{align*}
[X,Y_1] &= Y_2,\quad [X,Y_2]=Y_3;\\
[X,Z_1] &= Z_2;\\
[Z_1, Z_2] &= Y_3.
\end{align*}
Direct computation shows that an arbitrary symmetry of the standard
distribution of type $\m$ depends on 4 functions of 1 variable.

\subsection{Metabelian Lie algebras}
Let $\m=\m_{-1}\oplus\m_{-2}$ be a GNLA of depth 2. The structure of this Lie
algebra is completely determined by the skew-symmetric map:
$S\colon\wedge^2\m_{-1}\to \m_{-2}$ defined as a restriction of the Lie bracket
to $\m_{-1}$. Below we assume that $\m$ is non-degenerate.

Consider the dual map $S^*\colon \m_{-2}^*\to \wedge^2\m_{-1}^*$. Since $\m$ is
generated by $\m_{-1}$, the map $S^*$ is injective. Thus, up to the isomorphism
the structure of $\m$ is determined by the subspace $\im S^* \subset
\wedge^2\m_{-1}^*$. This subspace is considered up to the natural action of the
group $GL(\m_{-1})$ on $\wedge^2\m_{-1}^*$.

To simplify the notation, let $V=\m_{-1}$, $n=\dim V$, and let $P = \im S^*
\subset \wedge^2V^*$, $m=\dim P$. If we fix a basis in $V$, then $\wedge^2V^*$
can be identified with a space of all skew-symmetric $n$ by $n$ matrices, $P$
is spanned by $m$ matrices $B_1,\dots,B_m$ and the action of $GL(V)$ on
$\wedge^2V^*$ is given by:
\[
X.B = X^t B X,\quad X\in GL(V), B\in \wedge^2V^*.
\]

Let $y$ be an arbitrary element in $\m_{-1}=V$. It is easy to see that the rank
of $\ad y$ is the maximal number of linearly independent skew-symmetric
matrices $B$ in $P$, such that $i_y B \ne 0$. In more detail, let $P_y$ be a
subspace in $P$ defined as:
\[
P_y = \{ B \in P \mid i_y B = 0 \}.
\]
Then rank of $\ad y$ is equal to to codimension of $P_y$ in $P$.
\begin{thm}\label{thm:2step}
Let $\m=\m_{-1}\oplus\m_{-2}$ be a 2-step complex graded nilpotent Lie algebra
with $\dim\m_{-2}=2$. Then Tanaka prolongation of $\m$ is infinite-dimensional.
\end{thm}
\begin{proof}
Since the dimension of the Tanaka prolongation is preserved under extension of
the base field, we can assume that our base field is algebraically closed. Then
according to Theorem~\ref{thm:1} $\m$ is of infinite type if and only if there
exists a non-zero element $y\in \m_{-1}=V$ such that $P_y$ is one-dimensional.
Let $B_1$, $B_2$ be the basis of $P$. Then the condition $\dim P_y = 1$ is
equivalent to the existence of such $\lambda_1,\lambda_2$ that:
\[
(\lambda_1 B_1 + \lambda_2 B_2)y = 0.
\]
It is clear that this is equivalent to $\det (\lambda_1 B_1 + \lambda_2 B_2) =
0$, which always has non-trivial solutions in case of algebraically closed
field.
\end{proof}

Note that pencils of skew-symmetric matrices over an algebraically closed field
can be effectively classified using Kronecker results on pencils of matrices.
This has been done in the work of M.~Gauger~\cite{gauger}. Namely, any pair
$(A,B)$ of skew-symmetric matrices can be written as one matrix $P = \mu A+
\lambda B$, whose entries are linear forms in $\lambda$ and $\mu$. Such
matrices are classified by the following data:
\begin{itemize}
\item minimal indices $0\le m_1\le m_2\le m_p$, $p\ge0$ (in particular, the set of minimal indices can be empty);
\item elementary divisors $(\mu+a_1\lambda)^{e_1}$, \dots,
$(\mu+a_q\lambda)^{e_r}$, $(\lambda)^{f_1}$, \dots, $(\lambda)^{f_s}$ (each of
these divisors appears twice).
\end{itemize}
The canonical form of the pencil $P$ for this data is:
\begin{equation}\label{eq1}
P = \begin{pmatrix} \cM_{m_1} & & & & & & & & \\
 & \ddots & & & & & & & \\
 & & \cM_{m_p} & & & & & & \\
 & & & \cE_{e_1}(a_1) & & & & & \\
 & & & & \ddots & & & & \\
 & & & & & \cE_{e_r}(a_r) & & & \\
 & & & & & & \cF_{f_1} & & \\
 & & & & & & & \ddots & \\
 & & & & & & & & \cF_{f_s}
\end{pmatrix}
\end{equation}
where
\begin{align*}
\cM_m & = \begin{pmatrix} 0 & M_m \\
-M_m^t & 0
\end{pmatrix},
\quad (2m+1)\times(2m+1), \cM_0 = (0);\\
\cE_n(a) &= \begin{pmatrix} 0 & E_n(a) \\
-E_n(a)^t & 0
\end{pmatrix},
\quad (2n)\times(2n);\\
\cF_n &= \begin{pmatrix} 0 & F_n \\
-F_n^t & 0
\end{pmatrix},
\quad (2n)\times(2n);\\
\end{align*}
and
\begin{align*}
M_m &= \begin{pmatrix}
&&&&&& \lambda \\
&&&&& \lambda & \mu \\
&&&&\cdot &\mu & \\
&&& \cdot &\cdot && \\
&& \cdot &\cdot &&& \\
& \lambda & \cdot &&&& \\
\lambda & \mu &&&&& \\
\mu &&&&&&
\end{pmatrix},
\quad (m+1)\times m;\\
E_n(a) &=
\begin{pmatrix}
&&&&& \mu + a\lambda \\
&&&& \cdot & \lambda \\
&&& \cdot && \\
&& \cdot && \cdot & \\
&&& \cdot && \\
& \mu+a\lambda & \cdot &&& \\
\mu+a\lambda & \lambda  &&&&
\end{pmatrix},
\quad n \times n;\\
F_n &=
\begin{pmatrix}
&&&&& \lambda \\
&&&& \cdot & \mu \\
&&& \cdot && \\
&& \cdot && \cdot & \\
&&& \cdot && \\
& \lambda & \cdot &&& \\
\lambda & \mu  &&&&
\end{pmatrix},
\quad n \times n;\\
\end{align*}
In addition, the elementary divisors are considered up to non-degenerate linear
transformations of $(\lambda, \mu)$.

Note also, that the pencil $P$ above corresponds to a non-degenerate 2-step
nilpotent Lie algebra if and only if there are no minimal indices $0$ among the
set of minimal indices $m_1,\dots,m_p$. As we assume that these indices are
ordered, this is equivalent to the condition $m_1>0$.

Define the following subgroups in $GL(V)$ and the corresponding subalgebras in
$\gl(V)$:
\begin{align*}
\Aut(P) &= \{ X\in GL(V) \mid X^tBX \subset P \text{ for any }B\in P\},\\
H_0 &= \{X \in GL(V) \mid X^tBX=B \text{ for any }B\in P\},\\
\Der(P) &= \{ X\in \gl(V) \mid X^tB+BX \subset P \text{ for any }B\in P\},\\
\h_0 &= \{ X\in \gl(V) \mid X^tB+BX = 0 \text{ for any }B\in P\}.
\end{align*}

It is clear that:
\begin{itemize}
\item $\Der(P)$ and $\h_0$ are subalgebras in $\gl(V)$ corresponding to the
subgroups $\Aut(P)$ and $H_0$ respectively;
\item $H_0$ is a subgroup in $\Aut(P)$, $\h_0$ is a subalgebra in $\Der(P)$;
\item $\Aut(P)$ is naturally identified with the the group $\Aut_0(\m)$ of all
grading-preserving automorphisms of the Lie algebra $\m$ and $\Der(P)$ is
identified with the Lie algebra $\Der_0(\m)$ of all grading-preserving
derivations of $\m$;
\item $\h_0$ can be identified with all elements in $\Der_0(\m)$ that act
trivially on $\m_{-2}$.
\end{itemize}

Let us show explicitly that $\h_0$ always contains a rank 1 element. This will
give another proof that the standard prolongation of $\h_0$ is
infinite-dimensional.

Let $Q$ be any of blocks that appear on the diagonal in the canonical
form~\eqref{eq1} of the pencil $P$. That is $Q$ is one of the following
matrices: $\cM_m$ for $m\ge0$, $\cE_{e}(a)$ for $e\ge1$, or $\cF_{f}$, $f\ge1$.
We shall call $Q$ an elementary subpencil of the pencil $P$ and denote by
$\h_0(Q)$ the Lie algebra
\[
\h_0(Q) = \{ X\in \gl(k,\C)\mid X^tB+BX=0\text{ for all }B\in Q\}.
\]
It is easy to see that $\h_0(Q)$ can be embedded as a subalgebra into $\h_0$.

So, to prove the theorem it is sufficient to show that $\h_0(Q)$ contains
rank~1 element for each elementary subpencil. Moreover, since elementary
divisors are considered up to linear transformations of $(\lambda, \mu)$, we
can restrict ourselves only to elementary pencils $\cM_m$, $m\ge 1$ and
$\cF_r$, $r\ge1$.

In both these cases the subalgebra $\h_0(Q)$ is easily computed. Namely, denote
by $S_{k,l}(\C)$ the set of $k\times l$ matrices $(a_{ij})$ such that
$a_{ij}=a_{i+1,j+1}$ for any $1\le i<k$, $1\le j < l$. Denote also by $T_k(\C)$
the set of all $k\times k$ upper triangular matrices that also lie in
$S_{k,k}(\C)$. Note that both $S_{k,l}(\C)$ and $T_k(\C)$ contain elements of
rank~1 (for example, in the upper right corner).

We have:
\begin{align*}
\h_0(\cM_m) &= \left.\left\{ \begin{pmatrix} x E_{m+1} & Y \\ 0 & -x E_m
\end{pmatrix}\,\right|\, x\in\C, Y\in S_{m+1,m}(\C) \right\};
\\
\h_0(\cF_r) &= \left.\left\{ \begin{pmatrix} Y_{11} & Y_{12} \\ Y_{21} &
-Y_{11} \end{pmatrix}\,\right|\, Y_{11}, Y_{12}, Y_{21} \in T_{r}(\C) \right\};
\end{align*}
As we see, in both cases $\h_0(Q)$ contains a nilpotent element of rank 1.

\subsection{Remarks}
1. It follows form the proof of Theorem~\ref{thm:2step} that $\h_0$ contains
the block diagonal subalgebra, where each block consists of all elements in
$\h_0(Q)$ for the primitive subpencils $Q$. However, $\h_0$ can be larger than
this direct sum. For example, this happens in the case when there no minimal
indices and the set of $(\lambda)$ and $(\mu)$ where each of these divisors
appears with multiplicity $2p$ and $2q$ respectively. In this case $\m$ is
isomorphic to the direct sum of two Heisenberg lie algebras of dimension $2p+1$
and $2q+1$. The subalgebra $\h_0$ is isomorphic to
$\spp(2p,\C)\oplus\spp(2q,\C)$, while $\h_0(Q)$ for each primitive $Q$ is just
$\sll(2,\C)$.

2. It is well-known that there exist $2$-step nilpotent Lie algebras with
3-dimensional center, whose Tanaka prolongation is finite-dimensional. The
simplest example is a free 2-step nilpotent Lie algebra with 3-dimensional set
of generators and the 3-dimensional center:
\begin{align*}
\m_{-1} = \langle X_1, X_2, X_3 \rangle,\\
\m_{-2} = \langle X_{12}, X_{13}, X_{23} \rangle,
\end{align*}
where $[X_i,X_j]=X_{ij}$ for $1\le i < j \le 3$. Its Tanaka prolongation is
21-dimensional and is isomorphic to $\mathfrak{so}(7,\C)$. In this case the
subalgebra $\h_0$ is trivial.

We can also generalize this example for arbitrary number of generators. Namely,
let:
\begin{align*}
\m_{-1} = \langle X_1, X_2, \dots, X_k \rangle,\\
\m_{-2} = \langle Y_1, Y_2, Y_3 \rangle,
\end{align*}
where all non-zero Lie brackets have the form:
\begin{gather*}
[X_1,X_k]=[X_2,X_{k-1}]= \dots = Y_1;\\
[X_1,X_{k-1}]=[X_2,X_{k-2}]= \dots = Y_2;\\
[X_2,X_k]=[X_3,X_{k-1}]= \dots = Y_3.
\end{gather*}
Direct computation shows that Tanaka prolongation $\g(\m)$ of $\m$ has the
form:
\begin{itemize}
\item for $k=3$, $\g(\m)\equiv \mathfrak{so}(7,\C)$;
\item for $k=4$, $\g(\m)\equiv \mathfrak{sp}(6,\C)$;
\item for $k=5$, $\g(\m)=\m\oplus\g_0$, where $\g_0\equiv
\mathfrak{gl}(2,\C)$ and $\g_0$-module $\g_{-1}$ is irreducible;
\item for $k\ge 6$ and even, $\g(\m)=\m\oplus\g_0$, where
$\g_0\equiv\mathfrak{gl}(2,\C)\times \C$ and $\g_0$-module $\g_{-1}$ splits
into the sum of $k/2$ $\mathfrak{gl}(2,\C)$-modules $\langle X_i,
X_{k/2+i}\rangle$, $i=1,\dots,k/2-1$;
\item for $k\ge 7$ and odd, $\g(\m)=\m\oplus\g_0$, where $\g_0=\C^2$ and the
action of $\g_{0}$ on $\g_{-1}$ diagonalizes in the basis $\{X_1,X_2,\dots,
X_n\}$.
\end{itemize}

3. The above example of 2-step nilpotent Lie algebra with 3-dimensi\-o\-nal
center shows that the subalgebra $\h_0$ is trivial for generic 2-step nilpotent
Lie algebras with $\dim \m_{-2}\ge 3$. In particular, a generic GNLA of depth 2
with at least 3-dimensional center is of finite type. However, it is still an
open problem to classify all 2-step nilpotent Lie algebras with
infinite-dimensional Tanaka prolongation.

4. Even though the theorem above is formulated, and proved over $\C$, it is in
fact true over any field of characteristic $0$, and, in particular, over $\R$.
Indeed, the results of Kronecker and, thus, Gauger hold over any algebraically
closed filed of characteristic $0$. And it is easy to see that the dimension of
Tanaka prolongation does not change if we extend the field of scalars.

5. Let $D$ be a vector distribution of codimension 2 on a smooth manifold $M$,
such that $D^2=TM$. Then Theorem~\ref{thm:2step} implies that the symbol
$\m(x)$ of $D$ is of infinite type for each $x\in M$. However, this does not
imply that such distributions always have infinite-dimensional symmetry
algebras. For example, such distributions with the symbol corresponding to the
pencil $\cM_m$ are treated in~\cite{zelkryn}. It is proved that if a subbundle
$D'$ generated by all weak characteristics of $D$ is bracket generating, then
the symmetry algebra of $D$ is finite-dimensional. Note that in case of the
standard distribution $D$ of type $\m$, where $\m$ corresponds to the pencil
$\cM_m$, the subdistribution $D'$ is involutive.

\end{document}